\documentclass[final,1p,times]{elsarticle}

\usepackage{amssymb}
 \usepackage{amsthm}
\usepackage{amscd}
\usepackage{amsmath}
\usepackage{amsfonts}
\usepackage{amssymb}
\usepackage{multirow}
\usepackage{graphicx}
\usepackage{xcolor}
\newtheorem{theorem}{Theorem}

\newtheorem{Lemma}[theorem]{Lemma}

\usepackage{mathrsfs}
\usepackage{titletoc}
\usepackage{float}
\usepackage{flushend}
\journal{******}

\begin{document}

\begin{frontmatter}
\title{Existence results for
Kazdan-Warner type  equations on graphs}

\author{Pengxiu Yu}
\ead{Pxyu@ruc.edu.cn}


\address{School of Mathematics,
Renmin University of China, Beijing 100872, P. R. China}

\begin{abstract}
In this paper, motivated by the work of Huang-Lin-Yau (Commun. Math. Phys. 2020),
Sun-Wang (Adv. Math. 2022) and Li-Sun-Yang (Calc. Var. Partial Differential Equations 2024),
we investigate the existence of Kazdan-Warner type equations on a finite connected graph, based on the theory of Brouwer degree. Specifically, we consider the equation
\begin{equation*}
-\Delta u=h(x)f(u)-c,
\end{equation*}
where $h$ is a real-valued function defined on the vertex set $V$, $c\in\mathbb{R}$ and
\begin{equation*}
f(u)=
\left(1-\displaystyle\frac{1}{1+u^{2n}}\right)e^u
\end{equation*}
with $n\in \mathbb{N}^*$.  Different from the previous studies,
the main difficulty in this paper is to show that the corresponding equation has only
three constant solutions, based on delicate
analysis and the connectivity of graphs,
which have not been extensively explored in previous literature.
\end{abstract}


\end{frontmatter}

\titlecontents{section}[0mm]
                       {\vspace{.2\baselineskip}}
                       {\thecontentslabel~\hspace{.5em}}
                        {}
                        {\dotfill\contentspage[{\makebox[0pt][r]{\thecontentspage}}]}
\titlecontents{subsection}[3mm]
                       {\vspace{.2\baselineskip}}
                       {\thecontentslabel~\hspace{.5em}}
                        {}
                       {\dotfill\contentspage[{\makebox[0pt][r]{\thecontentspage}}]}

\setcounter{tocdepth}{2}



\numberwithin{equation}{section}

\section{Introduction}\let\thefootnote\relax\footnotetext{
On behalf of all authors, the corresponding author states that there is no conflict of interest.}

Let $(\Sigma,g)$ be a two dimensional compact Riemannian manifold
and let $g$ and $\tilde{g}=e^{2\phi}g$ be two conformal metrics, where $\phi$  is a smooth function.   The corresponding Gaussian curvatures are denoted
by $K$ and $\widetilde{K}$, respectively.  Then we have the following equation
\begin{equation}\label{gs}
\Delta_g \phi=K-\widetilde{K}e^{\phi},
\end{equation}
where $\Delta_g$ is  the Laplace–Beltrami operator associated with the metric $g$. Let
$\Delta_g \psi=K-\overline{K}$, with
\begin{equation*}
\overline{K}=\frac{1}{\mathrm{vol}(\Sigma)}\displaystyle\int_{\Sigma} K dv_g.
\end{equation*}
Setting $u=2(\phi-\psi)$,  the equation (\ref{gs}) can be  transformed into
\begin{equation}\label{gs1}
\Delta_g u=2\overline{K}-(2\widetilde{K}e^{2\psi})e^u.
\end{equation}
To proceed, the general case of  the equation (\ref{gs1}) is of independent interest and
takes the following specific form:
\begin{equation}\label{kw0}
\Delta_g u=c-he^u,
\end{equation}
where $c$ is a constant and $h$ is some
prescribed function. The solvability of
(\ref{kw0})  depends on the sign of $c$.
To proceed,
let $\bar{h}$  denote the average value of $h$ on the manifold
$\Sigma$ and
Kazdan-Warner \cite{kazdanwarner}
provided satisfactory characterizations for the solvability of the Kazdan-Warner equation (\ref{kw0}):

(i) if $c=0$ and $h\nequiv0$, then (\ref{kw0})  has a solution if and only if
$h$ changes sign and $\bar{h}<0$;

(ii) if $c>0$, then (\ref{kw0})  has a solution if and only if the set $\{x:h>0\}$ is
not empty;

(iii) if $c<0$  and (\ref{kw0}) has a solution, then $\bar{h}<0$. For $\bar{h}<0$, there exists a constant $-\infty\leq c_h<0$ such that (\ref{kw0}) has a solution for
any $c_h<c<0$ and there is no solution for any $c<c_h$. Furthermore, $c_h=-\infty$ if and only if $h\leq0$ and $h\nequiv0$ on $\Sigma.$

We refer the reader to  \cite{Caffarellivortex,ChenScalar,chenliqua,chenligau,ding,dingjost,
NolascoNontopolo,RicciardiVortices}  for further information on the Kazdan-Warner problem, which has been extensively studied in these works.

In recent years,  various topics in the analysis on graphs have received a lot of attention. Grigor’yan-Lin-Yang \cite{yangkw} first researched  the Kazdan–Warner equation (\ref{kw0}) on the finite graphs. Namely, they considered the equation $\Delta u=c-he^u$, where $\Delta$ is discrete graph Laplacian
with $c\in\mathbb{R}$ and $h$ is a function defined on the vertices. Subsequently,
Ge \cite{geneggative}, Liu-Yang \cite{liushuangnegetive} and  Zhang-Chang \cite{zhangxiaonegetive}
studied the Kazdan-Warner equation on graphs for the negative case; Later, Ge \cite{geinfinite} generalized
the existence result to  infinite graphs
and Keller-Schwarz \cite{kellercompactifiable}
extended the equation to canonically compactifiable graphs. Moreover,  some other important works on graphs can be found in \cite{gaoChern,yangyamabe, yangnonlinear,huagroundstate,linCalculus,
pansign,yangNormalized,
shaologarithmic}.

In this paper, inspired by the works
of Huang-Lin-Yau \cite{linyau}, Li-Sun-Yang \cite{lidegree} and Sun-Wang \cite{sundegree},
it is interesting to consider the following
Kazdan-Warner type equation on graphs via the theory of Brouwer degree.
\begin{equation}\label{kw1}
-\Delta u=h(x)f(u)-c,
\end{equation}
where $h$ is a real function on $V$ and $f$ satisfies the asymptotic
condition $\displaystyle\frac{f(t)}{e^t}\rightarrow a>0$ as $t\rightarrow\pm\infty$.
Note that when computing the Brouwer degree, it is essential to identify the constant solutions of the corresponding equation.
However, the difficulty of the problem (\ref{kw1}) lies in that there is no
explicit expression for $f(u)$, despite the fact that $f(u)$  can be written in terms of a limit form involving $e^u$. There are various functions that satisfy this asymptotic condition,   making it challenging to determine the constant solutions.
For this reason, we first take a specific form of $f(u)$ to analyze the equation (\ref{kw1}).
Precisely, we
consider the equation
\begin{equation}\label{function}
-\Delta u(x)=h(x)\left(1-\displaystyle\frac{1}{1+u^{2n}}\right)e^u -c:= Q(u) ,
\end{equation}
where $h$ and $c$ are the same as those in equation (\ref{kw1}) with  $n\in \mathbb{N}^*$.
By calculating the Brouwer degree of (\ref{function}) in a case-by-case analysis, and utilizing the homotopic invariance along with  Kronecker existence theorem, we establish the existence of solutions to  (\ref{function}).
However,  even for the easy case in (\ref{function}), there are significant difficulties in our proof.

The  main difficulties are listed as below. Firstly, the theory of Brouwer
degree implies that finding  the constant solutions to the corresponding  equation  is a crucial  point.
Nevertheless, it  is not easy to determine the constant solutions of (\ref{function}) even in this particular forms. In order to have a clearer understanding of the equation (\ref{function}),
we simulate the figure of $Q(u)$ using the computer when $n=3$, $h=1$ and $c=0.06$,
that is,
\begin{figure}[H]
  \centering
  \includegraphics[width=10cm,height=6cm]{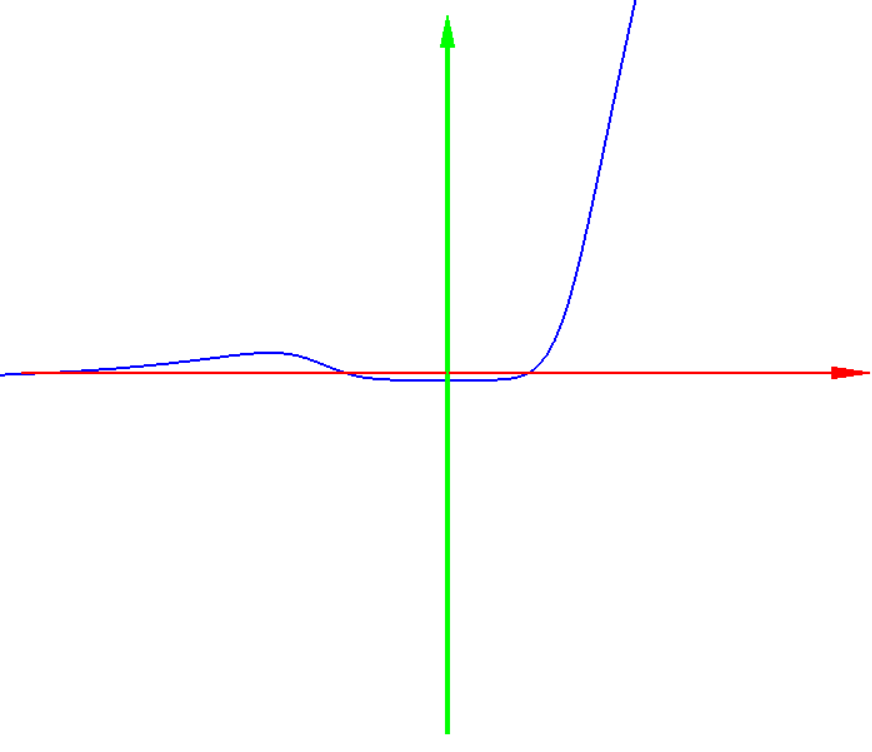}
  \caption{n=3, h=1, c=0.06}
  \label{figutrintroduction}
\end{figure}
\noindent Figure \ref{figutrintroduction} shows that if we take specific parameters, then  there exist three constant solutions to the corresponding equation of (\ref{function}), one of which is close to $-\infty$ and the remaining two solutions are close to $0$.
This is totally different from the previous work of Sun-Wang \cite{sundegree}, where the constant
solution is unique and can be obtained by a direct calculation.
Furthermore, although we cannot obtain an explicit expression for the constant solutions, the precise estimation of these constant solutions can be derived through careful analysis.
For the better understanding, we draw a schematic diagram of $Q(u)$ for a special case. Namely, when $n\in \mathbb{N}^*,  h=1,c=\varepsilon$, we have
\begin{figure}[H]
  \centering
  \includegraphics[width=10cm,height=6cm]{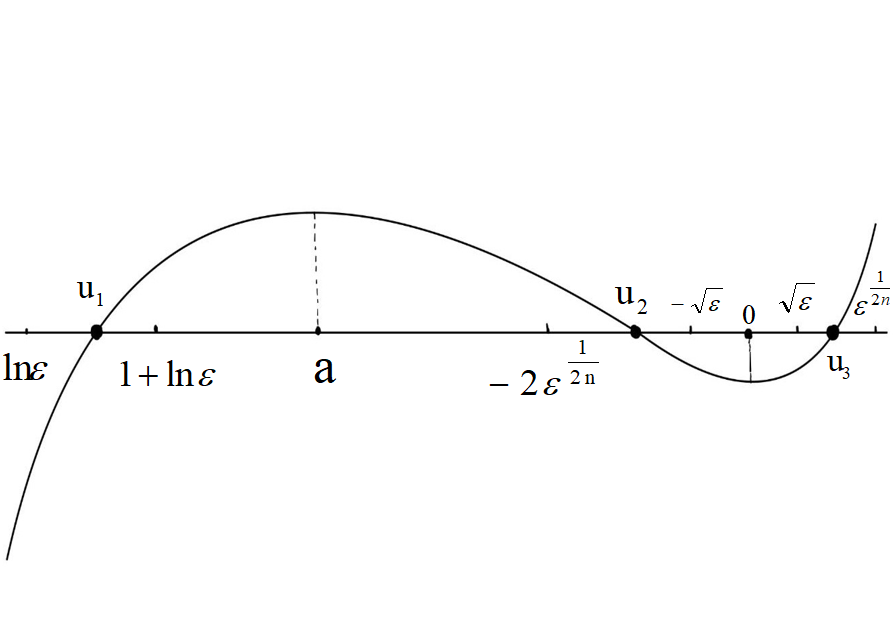}
 \caption{Q(u)}
\label{fintroduction}
\end{figure}

Figure \ref{fintroduction} implies that, under the given conditions,  there exist three distinct constant solutions for the function $Q(u)$. To proceed,  we denote these constant solutions as  $u_1$, $u_2$ and $u_3$, respectively,  and derive that
\begin{equation*}
u_1\in(\ln\varepsilon, 1+\ln\varepsilon),
\hspace{0.5cm}
u_2\in(-2\varepsilon^{\frac{1}{2n}}, -\sqrt{\varepsilon}),
\hspace{0.5cm}
u_3\in(\sqrt{\varepsilon},\varepsilon^{\frac{1}{2n}} ).
\end{equation*}
Moreover, we observe that  $u_1$ is close to $-\infty$ as $\varepsilon\rightarrow0^+$, while $u_2$ and $u_3$ are both close to  $0$, with  $u_2\leq 0$ and $u_3\geq0$.

Secondly, a significant challenge we face is to demonstrate that  the  corresponding equation of (\ref{function}) only has three constant solutions. Different from the lemma 4.4 derived by Huang-Lin-Yau \cite{linyau},  our paper
needs to obtain a more precise estimation of $u$. For more information, see section 4.
Notably, we use the connectivity of graphs, an uncommon approach in other literature. By utilizing the method outlined in our paper, we can deduce that the Chern-Simons Higgs model explored in \cite{lidegree} also has only two constant solutions, making our findings particularly interesting.

Finally, according to the above analysis, we calculate the corresponding Brouwer degree. Unlike the method described in \cite{lidegree,sundegree},
we are unable to  substitute the constant solutions into the calculations directly as we cannot obtain the explicit expression of the constant solutions in our paper.
Fortunately, we can derive the desired results by analyzing the monotonicity of $DQ(u)$.

Moreover, in this paper, demonstrating that the equation possesses only constant solutions is a crucial step,
which is more difficult than  previous
studies \cite{lidegree,sundegree} and of independent interest in other
equations.
For example, in \cite{weijunchengs6},
Gui-Li-Wei-Ye  conjectured that
$J_{\alpha}(u)\geq C(\alpha)$ can be chosen to be $0$, where the functional $J_{\alpha}(u)$  is respect to a $Q$-curvature-type equation. They proved that axially symmetric solutions to the $Q$-curvature type problem must be constants, then the conjecture is valid,
which is associated with Liouville theorem.
Hence, it is meaningful to study the related issues of constant solutions. These solutions provide insights into the nature and behavior of equations.


The remaining part of this paper is organized as follows: In Section 2,  we give some preliminaries  and state our main results.
In Section 3, we study the blow-up behaviour for the Kazdan-Warner type equation (\ref{kw1}). In Section 4, we calculate the
Brouwer degree for the equation (\ref{function}). Throughout this paper,
we will frequently use the notation $u(x)\in [a,b]$, which means
$u(x_i)\in [a,b]$ for any $i=1,\cdots,m$ with  $\#V=m$.

\section{Settings and main results}
Let $G=(V,E)$ be a finite connected  graph, where $V$  denotes the set of all vertices
and $E$ denotes  the set of  all edges. Let
 $\mu: V\rightarrow \mathbb{R}^+$ be a finite measure and we assume positive
symmetric weights $w_{xy}=w_{yx}$ on edges
$xy\in E$. The Laplace operator acting on $u$ reads as
\begin{equation*}\label{laplace}
\Delta u(x)=\frac{1}{\mu(x)}
\displaystyle\sum_{y\sim x}
w_{xy}(u(y)-u(x)),
\end{equation*}
where $y\sim x$ means $xy\in E$.
The corresponding gradient  is defined as
\begin{equation*}
\nabla u(x)
=
\left(
\sqrt{ \frac{\omega_{xy_1}}{2\mu(x)} } ( u(y_1)-u(x)  ) ,
\cdots,
\sqrt{\frac{\omega_{xy_l}}{2\mu(x)} }  ( u(y_l)-u(x)  )
\right),
\end{equation*}
where $l$ denotes the number of all neighbours of $x$. The integral of function $g$
is denoted by
\begin{equation*}
\int_V g d\mu
=\sum_{x\in  V}\mu(x)g(x),
\end{equation*}
and the average of this integral is
\begin{equation*}
\bar{g}
=
\frac{1}{|V|}\int_V g d\mu
=\frac{1}{|V|}\sum_{x\in  V}\mu(x)g(x),
\end{equation*}
where $|V|=\sum_{x\in V}\mu(x)$.

Next, we introduce some important
lemmas about topological degree.

\begin{Lemma}[Homotopic invariance \cite{chang}]\label{homo}
If $\phi: \bar{B}_R\times[0,1]\rightarrow V^{\mathbb{R}}$ is continuous and $p\notin(\partial B_R \times[0,1])$, then
\begin{equation*}
\mathrm{deg}(\phi({\cdot, t}), B_R, p)
=\mathrm{constant}.
\end{equation*}
\end{Lemma}

\begin{Lemma}[Kronecker's  existence theorem \cite{chang}]\label{kro}
If $P\notin F_{\rho, g}(\partial B_R )$ and
$\mathrm{deg}(F_{\rho, g}, B_R, p)\neq 0$, then
\begin{equation*}
F_{\rho, g}^{-1}(p)\neq \emptyset.
\end{equation*}
\end{Lemma}

Assume $f(u)=ae^u+g(u)e^u>0$ and $a>0$ is a positive constant. The function $g$ is uniformly bounded with respect to $u$ and
$g(u)\rightarrow0$ as $u\rightarrow\pm\infty$. Now, we consider the Kazdan-Warner type equation $-\Delta u=h(x)f(u)-c$, and our first result is

\begin{theorem}\label{thm1}
Let $G=(V,E)$ be a finite connected graph with positive measure $\mu$ and positive symmetric weights $w_{xy}$.
Assume $h\in V^{\mathbb{R}}$ and $c\in \mathbb{R}$ satisfy:\\
\noindent(1) if $c>0$, then there exists $x_0\in V$ such that $h(x_0)>0$,

\noindent(2) if $c=0$, then $h$ changes sign and $\int_V h d\mu<0$,

\noindent(3) if $c<0$, then there exists $x_0\in V$ such that $h(x_0)<0$.

\noindent Then there exists a constant $C$ depending on $h,c,G$ such that every solution
to
\begin{equation}\label{kw}
-\Delta u=h(x)f(u)-c
\end{equation}
satisfies
\begin{equation*}
\max_{x\in V} |u(x)|\leq C.
\end{equation*}
\end{theorem}

According to Theorem \ref{thm1}, the Brouwer degree is well defined.
Consider the following map $F:L^{\infty}(V)\rightarrow L^{\infty}(V)$
\begin{equation}\label{F}
F(u)=-\Delta u(x)-h(x)\left(1-\displaystyle\frac{1}{1+u^{2n}}\right)e^u+c,
\end{equation}
then we have the second result.

\begin{theorem}\label{degree}
Let $G=(V,E)$ be a finite connected graph, and the map $F$ is defined  as  in (\ref{F}),
then there exists a large number $R_0>0$ such that
\begin{equation*}
d_{h,c}:=\mathrm{deg}(F,B_R,0)
=
\left\{
\begin{array}{lll}
-1, \hspace{0.2cm}   &c\geq0;\\[12pt]
(-1)^m, \hspace{0.2cm} &c<0 \hspace{0.1cm} \mathrm{and} \hspace{0.1cm}h<0.
\end{array}
\right.
\end{equation*}
where
$B_R=\{u\in L^{\infty}(V):||u||_{L^{\infty}(V)}<R\}$  is a ball in $L^{\infty}(V)$ and $\#V=m$.
\end{theorem}

If the Brouwer degree is nonzero, then by the Kronecker existence theorem, we derive that there exists at least one solution.

\noindent{\bf Remark.}
According to Theorem \ref{degree},  we know that
when $c\geq0$  with $\forall n\in \mathbb{N}^*$ or $c<0$ with $h<0$ , we obtain that (\ref{function}) has at least one solution. However, it is difficult to analyze  how many specific solutions of this equation.
Based on Lemma \ref{degree1}, there are perhaps three constant solutions to (\ref{function}), which is interesting to consider this problem in the future.

\section{Blow-up analysis}

In this section, we study the blow-up
behavior for the Kazdan-Warner type equation
(\ref{function}).

\begin{Lemma}\label{brezis}
Assume $G=(V,E)$ be a connected finite graph and  $c_n\in V^{\mathbb{R}}$, $h_n\in V^{\mathbb{R}}$ satisfy
\begin{equation*}
\begin{array}{ccc}
\lim\limits_{n\rightarrow\infty} c_n=c, \hspace{0.2cm}
\lim\limits_{n\rightarrow\infty} h_n(x)=h(x), \hspace{0.2cm} \forall x\in V.
\end{array}
\end{equation*}
If $u_n\in V^{\mathbb{R}}$ be a sequence of solutions to
\begin{equation}\label{1}
-\Delta u_n = h_n(x) f(u_n(x)) - c_n, \hspace{0.2cm} x\in V.
\end{equation}
Then we have the following alternatives:

\noindent(1) either $u_n$ is uniformly bounded, or

\noindent(2) $u_n$ converges to $-\infty$ uniformly, or

\noindent(3) there exists $x_0\in V$ such that $h(x_0)=0$
and $u_n(x_0)$ converges to $+\infty$. Furthermore, $u_n$ is uniformly bounded in
$\{x\in V:h(x)>0\}$,  but $u_n$ is only uniformly bounded from below in $V$.
\end{Lemma}

\begin{proof}
If $u_n$ is uniformly bounded from above, then combining $f(u_n)=ae^{u_n}+g(u_n)e^{u_n}$ and $|g(u_n)|\leq C$ leads to $f(u_n)\leq C $. Clearly, we have $|\Delta u_n|\leq C$ from the equation (\ref{1}). Applying the elliptic estimate (Lemma 3.2 in \cite{sundegree}),
there is $\max_V u_n-\min_V u_n\leq C\max_V |\Delta u_n|\leq C$.
We can see that if $\min_V u_n$ is uniformly bounded from below, then (1) holds
and if
$\liminf\limits_{n\rightarrow\infty}\min_V u_n=-\infty$, then (2) holds.

Now, we need to prove (3).  If $\limsup\limits_{n\rightarrow\infty}u_n=+\infty$, then we assume that there exists  some $x_0\in V$ such that
$0<u_n(x_0)=\max_V u_n\rightarrow +\infty$ as $n\rightarrow+\infty$. To proceed, we utilize Kato's inequality (Lemma 3.3 in \cite{sundegree})
to obtain
\begin{equation*}
\begin{array}{lll}
-\Delta u_n^-
&=-\Delta(-u_n)^+ \\[4pt]
&\leq-\chi_{\{-u_n>0\}}\Delta(-u_n)\\[4pt]
&=\chi_{\{u_n>0\}}\Delta(u_n)\\[4pt]
&\leq c_n^+ + Ch_n^-,
\end{array}
\end{equation*}
therefore,
\begin{equation*}
\begin{array}{lll}
||\Delta u_n^-||_{L^1(V)}
&=
\displaystyle\int_V |\Delta u_n^-| d\mu\\[12pt]
&=
\displaystyle\int_{\Delta u_n^-\geq 0 } \Delta u_n^- d\mu
-
\displaystyle\int_{\Delta u_n^-< 0 } \Delta u_n^- d\mu\\[12pt]
&=
-2\displaystyle\int_{\Delta u_n^-< 0 } \Delta u_n^- d\mu\\[12pt]
&\leq 2\displaystyle\int_{\Delta u_n^-< 0 } ( c_n^+ + Ch_n^- ) d\mu\\[12pt]
&\leq C.
\end{array}
\end{equation*}
From the inequality above and  the elliptic estimate (Lemma 3.2 in \cite{sundegree}), we deduce that
\begin{equation*}
\max_V u_n^-=\max_V u_n^- - \min_V u_n^-\leq C\max_V |\Delta u_n^-|\leq C,
\end{equation*}
thereby showing that $u_n$ is uniformly bounded from below.

According to the result above, we deduce that for every $x_1\in V$,
\begin{equation*}
h_n(x_1) f(u_n(x_1)) - c_n=-\Delta u_n(x_1)\leq C(u_n(x_1)+1),
\end{equation*}
which gives
\begin{equation}\label{01}
\begin{array}{lll}
h_n(x_1)
&\leq
C(u_n(x_1)+1)f(u_n(x_1))^{-1}\\[2pt]
&\leq
C(u_n(x_1)+1)\displaystyle\frac{2}{a}e^{-u_n(x_1)}.
\end{array}
\end{equation}
Taking $n\rightarrow\infty$ in (\ref{01}),  there holds $h(x_1)\leq 0$. Meanwhile, this also shows that
$u_n$ is uniformly bounded in $\{x\in V:h(x)>0\}$.

Now, it remains to validate $h(x_0)=0$.   Noticing that  $u_n(x_0)\rightarrow +\infty$
implies $h(x_0)\leq0$, we only need to prove $h(x_0) \geq0$.
Utilizing the maximum principle to get
\begin{equation*}
h_n(x_0) f(u_n(x_0)) - c_n=-\Delta u_n(x_0)\geq0.
\end{equation*}
Hence,
\begin{equation*}
h_n(x_0)) \geq c_n f(u_n(x_0))^{-1},
\end{equation*}
we obtain $h(x_0)\geq0$ by letting $n\rightarrow+\infty$,
which completes the proof of Lemma \ref{brezis}.
\end{proof}

Next, we show the compactness result.

\begin{Lemma}\label{compactness}
Assume $G=(V,E)$ be a connected finite graph and there exists a constant $A>0$ such that

\noindent (1) $\max_V (|h|+|c|)\leq A$;

\noindent (2) if $h(x)>0$ for some $x\in V$, then $h(x)\geq A^{-1}$;

\noindent (3) if $c>0$, then $c\geq A^{-1}$;

\noindent (4) if $c=0$, then $\int_V h d\mu\leq - A^{-1}$;

\noindent (5) if $c<0$,  then $c\leq -A^{-1}$ and $\min_V h\leq -A^{-1}$.

\noindent Then there exists a constant $C(A,G)\geq0$ such that every solution to $-\Delta u=h(x)f(u)-c$ satisfies
\begin{equation*}
\max_V |u(x)|\leq C.
\end{equation*}

\end{Lemma}

\begin{proof}
We prove this lemma by contradiction. Let $h_n$ and $c_n$ satisfy the condition (1)-(5) and
\begin{equation*}
\begin{array}{ccc}
\lim\limits_{n\rightarrow\infty} c_n=c, \hspace{0.2cm}
\lim\limits_{n\rightarrow\infty} h_n=h, \hspace{0.2cm}
\lim\limits_{n\rightarrow\infty} ||u_n||=\infty,
\end{array}
\end{equation*}
where $u_n\in V^{\mathbb{R}}$  is a sequence of solutions to
\begin{equation*}
-\Delta u_n = h_n(x) f(u_n(x)) - c_n.
\end{equation*}
Now, if $u_n\rightarrow-\infty$ uniformly as $n\rightarrow\infty$, then
\begin{equation}\label{2}
|-\Delta (u_n-\min_V u_n)|
=
|h_n(x) f(u_n(x)) - c_n|
\leq
c|f(u_n)|+C
\leq C.
\end{equation}
Combining (\ref{2}) and the maximum principle (Lemma 3.1 in \cite{sundegree}) leads to
\begin{equation*}
\max_V (u_n-\min_V u_n)\leq C.
\end{equation*}
Without losing generality, we assume that $u_n-\min_V u_n $ converges  to $w$, which satisfies
the equation $-\Delta w=-c$ and $\min_V w=0$. This  thereby implies that $c=0$ and $w=0$.
Due to assumptions (3) and (5), we have $c_n=0$. Subsequently, following assumption (4), we
obtain $\int_V h d\mu\leq - A^{-1}$.
 However, a direct computation shows that
\begin{equation*}\label{3}
\begin{array}{lll}
0&=e^{-\min_V u_n}\displaystyle\int_V h_n f(u_n) d\mu \\[10pt]
&=
e^{-\min_V u_n}\displaystyle\int_V h_n(ae^{u_n}+g(u_n)e^{u_n})d\mu\\[10pt]
&=
\displaystyle\int_V ah_ne^{u_n-\min_V u_n }d\mu + \displaystyle\int_V h_n g(u_n)e^{u_n-\min_V u_n }d\mu\\[10pt]
&\rightarrow a\displaystyle\int_V he^w d\mu,
\end{array}
\end{equation*}
as $n\rightarrow\infty$, which is a contradiction,
 after in view of
$w=0$ and assumption
(\ref{3}) (which immediately implies  $ A^{-1}\leq0$).

According to the above argument and Lemma \ref{brezis}, we may assume that $\max_V u_n\rightarrow+\infty$. Then $u_n$ is uniformly bounded in $\Omega=\{x\in V:h(x)>0\}$.  However,  $u_n$ is only uniformly bounded from below in $V$ and the set $\{x\in V:h(x)=0\}$ is non-empty. If $n$ is large, then based on assumption (2) and noticing that the following sets are equivalent, namely,
\begin{equation}\label{omega}
\Omega=\{x\in V:h_n(x)>0\}=\{x\in V:h_n(x)\geq A^{-1}\}=\{x\in V:h(x)\geq A^{-1}\}.
\end{equation}
From the equation (\ref{1}), we get
\begin{equation*}
\begin{array}{lll}
\displaystyle\int_V c_n d\mu
&=\displaystyle\int_V h_nf(u_n)d\mu \\[12pt]
&=
\displaystyle\int_{ \{h_n(x)>0\} } h_nf(u_n)d\mu + \displaystyle\int_{ \{h_n(x)\leq0\} } h_nf(u_n)d\mu\\[12pt]
&\leq
C-\displaystyle\int_{ \{h_n(x)\leq0\} } h_n^{-}f(u_n)d\mu,
\end{array}
\end{equation*}
which deduces
\begin{equation*}
\displaystyle\int_{ V } h_n^{-}f(u_n)d\mu
\leq C-\displaystyle\int_V c_n d\mu\leq C.
\end{equation*}
This combined with the equation (\ref{1}) implies that $||\Delta u_n||_{L^1(V)}\leq C$, then it follows from the maximum principle (Lemma 3.1 in \cite{sundegree}) that
\begin{equation*}
\max_V u_n \leq \min_V u_n + C.
\end{equation*}
In addition, according to the assumption $\max_V u_n\rightarrow+\infty$, we obtain
$u_n$ converges to $+\infty$ uniformly, which implies  $\Omega=\emptyset$. Hence, we assume $h_n\leq0$ and there is $\int_V c_n d\mu\leq0$. When $c_n=0$, we have $h_n=0$, which contradicts the assumption (4). Consequently, there holds $c_n<0$. Due to the assumption (5), we get
\begin{equation*}
-CA
\leq\displaystyle\int_V c_nd\mu
=\displaystyle\int_V h_n f(u_n)  d\mu
\leq -A^{-1} \mu_{min} (f(u_n))_{min},
\end{equation*}
which gives $\min_V u_n\leq C$. This is a contradiction. Consequently, we complete the proof of Lemma \ref{compactness}.
\end{proof}

As a result of Lemma \ref{compactness}, we obtain Theorem  \ref{thm1}.

\section{Brouwer degree}

In this section, we shall prove  Theorem \ref{degree}. Precisely, we will calculate
the topological Brouwer  degree of certain maps related to the Kazdan-Warner type equation.

\begin{Lemma}\label{degree1}
If $c>0$, then for the connected finite graph $G=(V,E)$  with $\max_V h>0$, we have $d_{h,c}=-1$.
\end{Lemma}

\begin{proof}
For $n\in \mathbb{N^*}$,  we let $u_t\in V^{\mathbb{R}}$ be a sequence of solutions to
\begin{equation*}
-\Delta u_t
=
\left(h^+-(1-t)h^-\right)\frac{u_t^{2n} e^{u_t}}{1+u^{2n}_t}-(1-t)c-t\varepsilon, \hspace{0.2cm} \forall t\in [0,1],
\end{equation*}
where $\varepsilon>0$ is small to be determined. As an application of the Lemma \ref{compactness}, we know that $u_t$ is uniformly bounded. Then due to the homotopic invariance (Lemma \ref{homo}), we can assume $h\equiv1$ and $c=\varepsilon>0$. Now, considering the equation
\begin{equation}\label{4}
-\Delta u=\displaystyle\frac{u^{2n}}{1+u^{2n}}e^u -\varepsilon,
\end{equation}
we calculate the Brouwer degree of $F(u):=-\Delta u-\frac{u^{2n}}{1+u^{2n}}e^u +\varepsilon $ by three steps.
Firstly, we find the corresponding constant solutions; Secondly, we prove that there are only three constant solutions to the equation (\ref{4}); Finally, we calculate the Brouwer degree.

{\bf Step 1.} Set $Q(u):=\displaystyle\frac{u^{2n}}{1+u^{2n}}e^u -\varepsilon$, an easy calculation leads to
\begin{equation*}
Q'(u)
=
\displaystyle\frac{2nu^{2n-1}+u^{2n}(1+u^{2n})}{(1+u^{2n})^2}e^u
=\displaystyle\frac{u^{2n-1}(u^{2n+1}+u+2n)} {(1+u^{2n})^2}e^u.
\end{equation*}
Let $Q'(u)=0$. Observing  that $u=0$ is a constant solution to this equation. Further, we solve the equation
\begin{equation*}\label{5}
g(u):=u^{2n+1}+u+2n=0.
\end{equation*}
Noting that $g'(u)=(2n+1)u^{2n}+1>0$ with $g(-n)<0$ and $g(0)>0$,
we conclude that there exists only one solution, denoting by $a\in (-n,0)$, to the equation $g(u)$.
In other words,  $Q'(a)=0$. As a result,
the following table shows the monotonicity of  $Q(u)$, where $\uparrow$ represents $Q(u)$ is increasing and $\downarrow$ represents $Q(u)$ is decreasing.

\begin{table}[!htbp]
\begin{center}
\label{table}
\begin{tabular}
{|p{10mm}|p{15mm}|p{15mm}|p{15mm}|p{15mm}|p{15mm}|}
\hline
 &\centering $(-\infty,a)$  &\centering a   &\centering (a,0)
&\centering 0 &$(0,+\infty)$\\
\hline
\centering$Q'(u)$  & \centering $+$              &\centering0
&\centering $-$    &\centering0  &  $\hspace{15pt}+$ \\
\hline
\centering$Q(u)$   &\centering$\uparrow$      &
&\centering$\downarrow$   &   &$\hspace{15pt}\uparrow$  \\
\hline
\end{tabular}
\caption{Q(u)}
\end{center}
\end{table}
According to the above table, we know that the equation $Q(u)=0$ has three  constant solutions. Next, we need to give a precise estimate of these three constant solutions. \\
For $u=\ln\varepsilon$, a direct calculation yields
\begin{equation}\label{6}
\begin{array}{lll}
Q(\ln\varepsilon)
&=\displaystyle\frac{(\ln\varepsilon)^{2n}\cdot\varepsilon}{1+(\ln\varepsilon)^{2n}}-\varepsilon\\[10pt]
&=\displaystyle\frac{-\varepsilon }{1+(\ln\varepsilon)^{2n}}\\[10pt]
&<0.
\end{array}
\end{equation}
For $u=1+\ln\varepsilon$, there holds
\begin{equation}\label{7}
\begin{array}{lll}
Q(1+\ln\varepsilon)
&=
\displaystyle
\frac{ (1+\ln\varepsilon)^{2n}\cdot e\cdot\varepsilon-\varepsilon- (1+\ln\varepsilon)^{2n}\cdot\varepsilon }{1+ (1+\ln\varepsilon)^{2n}}\\[10pt]
&=
\displaystyle\frac{\varepsilon\cdot ( (1+\ln\varepsilon)^{2n}(e-1)-1 ) }{1+ (1+\ln\varepsilon)^{2n} }>0.
\end{array}
\end{equation}
Hence, from  (\ref{6}) and (\ref{7}), we derive that there exists a solution
$u_1\in(\ln\varepsilon, 1+\ln\varepsilon)$. Similarly,\\
For $u=-\sqrt{\varepsilon}$, we have
\begin{equation}\label{8}
\begin{array}{lll}
Q(-\sqrt{\varepsilon} )
&=\displaystyle\frac{\varepsilon^n}{1+\varepsilon^n}e^{-\sqrt{\varepsilon}}-\varepsilon\\[12pt]
&=\displaystyle\frac{\varepsilon^n}{1+\varepsilon^n}(1-\sqrt{\varepsilon})-\varepsilon+o_{\varepsilon}(1)\\[12pt]
&=\displaystyle\frac{\varepsilon\cdot(\varepsilon^{n-1}-\varepsilon^n-\varepsilon^{n-\frac{1}{2}}-1)}{1+\varepsilon^n}
+o_{\varepsilon}(1) \\[12pt]
&<0,
\end{array}
\end{equation}
here we used $e^t=1+t+o_t(1)$, where $o_t(1)\rightarrow0$ as $t\rightarrow0$. \\
For $u=-2\varepsilon^{\frac{1}{2n}}$, an easy calculation leads to
\begin{equation}\label{9}
\begin{array}{lll}
Q(-2\varepsilon^{\frac{1}{2n}})
&=\displaystyle\frac{2^{2n}\varepsilon}{1+2^{2n}\varepsilon}e^{ -2\varepsilon^{\frac{1}{2n}} }-\varepsilon\\[12pt]
&=\displaystyle\frac{2^{2n}\varepsilon}{1+2^{2n}\varepsilon}
(1-2\varepsilon^{\frac{1}{2n}})-\varepsilon+o_{\varepsilon}(1)\\[12pt]
&=\displaystyle\frac{\varepsilon(2^{2n}-2^{2n+1}\varepsilon^{\frac{1}{2n}}-2^{2n}\varepsilon)}{1+2^{2n}\varepsilon}
+o_{\varepsilon}(1)\\[12pt]
&>0.
\end{array}
\end{equation}
It follows from (\ref{8})-(\ref{9}) that there exists a solution
$u_2\in( -2\varepsilon^{\frac{1}{2n}}, -\sqrt{\varepsilon} )$.\\
For $u=\sqrt{\varepsilon}$, a direct  computation gives
\begin{equation}\label{10}
\begin{array}{lll}
Q(\sqrt{\varepsilon})
&=\displaystyle\frac{\varepsilon^n}{1+\varepsilon^n}e^{\sqrt{\varepsilon}}-\varepsilon\\[10pt]
&=\displaystyle\frac{\varepsilon^n}{1+\varepsilon^n}(1+\sqrt{\varepsilon})-\varepsilon +o_{\varepsilon}(1)\\[10pt]
&=\displaystyle\frac{\varepsilon(\varepsilon^{n-1}+\varepsilon^{n-\frac{1}{2}} - \varepsilon^n -1) } {1+\varepsilon^n}+o_{\varepsilon}(1)\\[10pt]
&<0.
\end{array}
\end{equation}
For $u=\varepsilon^{\frac{1}{2n}}$, we get
\begin{equation}\label{11}
\begin{array}{lll}
Q(\varepsilon^{\frac{1}{2n}})
&=
\displaystyle\frac{\varepsilon}{1+\varepsilon}e^{ \varepsilon^{\frac{1}{2n}} }-\varepsilon\\[10pt]
&=\displaystyle\frac{\varepsilon^{\frac{1}{2n}+1}-\varepsilon^2}{1+\varepsilon}+o_{\varepsilon}(1)\\[10pt]
&>0.
\end{array}
\end{equation}
Combining (\ref{10}) and (\ref{11}), we obtain that there exists a solution $u_3\in(\sqrt{\varepsilon},\varepsilon^{\frac{1}{2n}} ).$

Due to the fact that the parameter $\varepsilon$ is sufficiently small and $n$ is an arbitrary positive integer, the computer cannot simulate a figure of the function $Q(u)$ accurately. However, for better comprehension, we provide an example with $n=3$ and $\varepsilon=0.06$. The figure of $Q(u)$ is as followings:

\begin{figure}[H]
  \centering
  \includegraphics[width=10cm,height=6cm]{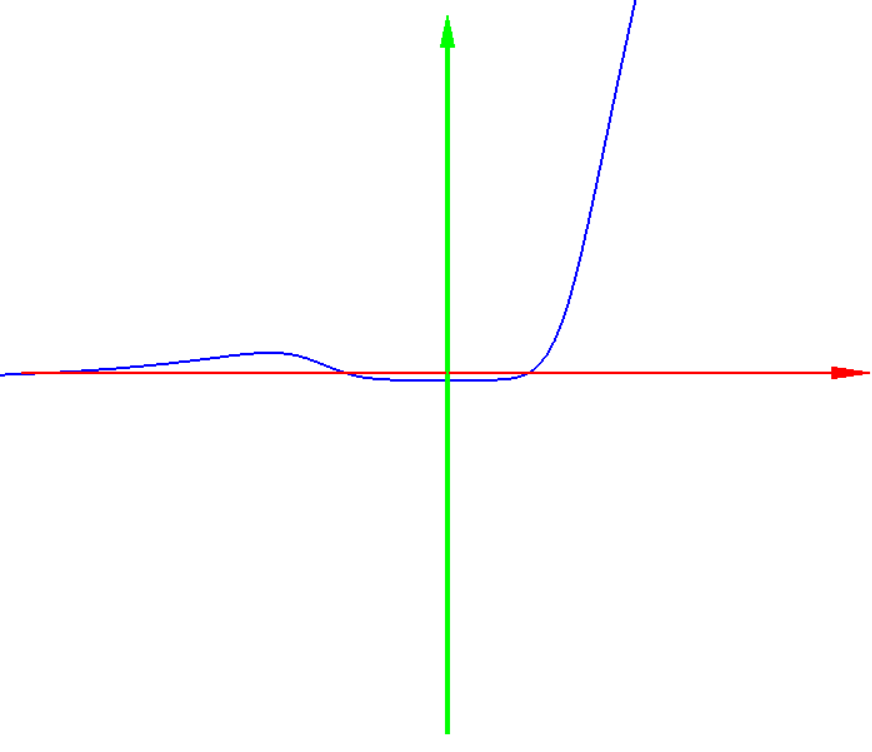}
  \caption{Example}
\end{figure}

\noindent As shown in the figure,  the equation $Q(u)=0$ has three constant solutions in this case.
This result is consistent with our analysis. Moreover, from the discussion above, we know that the constant solutions of $Q(u)=0$ are
\begin{equation}\label{constantsolution}
u_1\in(\ln\varepsilon, 1+\ln\varepsilon),
\hspace{0.5cm}
u_2\in(-2\varepsilon^{\frac{1}{2n}}, -\sqrt{\varepsilon}),
\hspace{0.5cm}
u_3\in(\sqrt{\varepsilon},\varepsilon^{\frac{1}{2n}} ),
\end{equation}
respectively. For convenience,  we have drawn a schematic diagram to further understand the properties of $Q(u)$. Namely,
\begin{figure}[H]
  \centering
  \includegraphics[width=10cm,height=6cm]{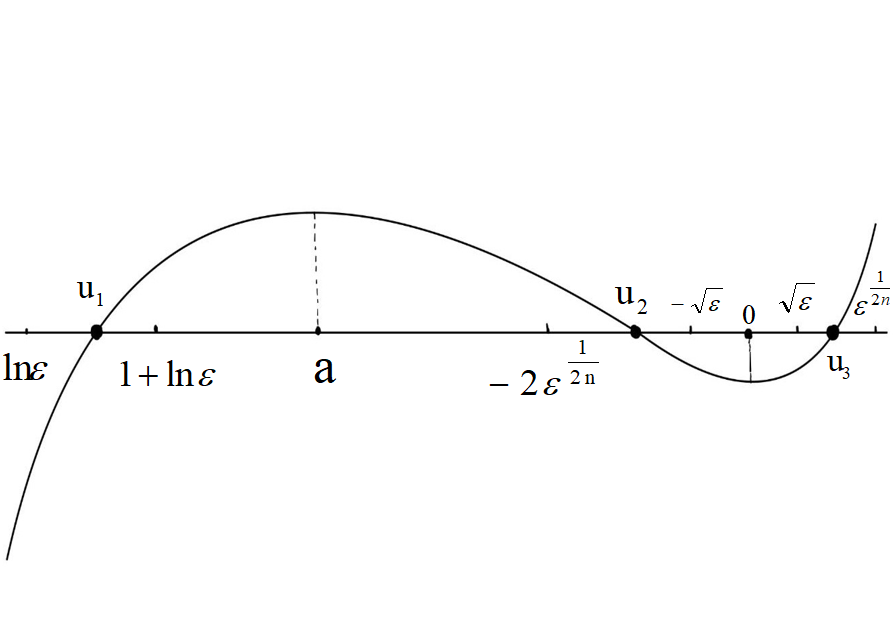}
  \label{12}\caption{Q(u)}
\end{figure}

{\bf Step 2.} Next, we are ready to show that the equation $Q(u)=0$ only has constant solutions.
If $u$ satisfies (\ref{4}), then integration by parts gives
\begin{equation*}
0=-\int_V\Delta u d\mu=\displaystyle\int_V\frac{u^{2n}}{1+u^{2n}}e^u d\mu-\int_V\varepsilon d\mu,
\end{equation*}
this implies
\begin{equation*}\label{13}
\begin{array}{lll}
\varepsilon |V|
=\displaystyle\int_V\frac{u^{2n}}{1+u^{2n}}e^u d\mu
=\displaystyle\sum_{i=1}^m\mu(x_i)\displaystyle\frac{u(x_i)^{2n}}{1+u(x_i)^{2n}}e^{u(x_i)}
\geq \mu_{\min} \displaystyle\frac{u(x_i)^{2n}}{1+u(x_i)^{2n}}e^{u(x_i)},
\end{array}
\end{equation*}
and hence
\begin{equation}\label{14}
K(u(x_i)):=\displaystyle\frac{u(x_i)^{2n}}{1+u(x_i)^{2n}}e^{u(x_i)}-\displaystyle\frac{\varepsilon|V|}{ \mu_{\min}}\leq0.
\end{equation}
As proved in step 1, we obtain that there exists three solutions to $K(u(x_i))=0$, that is
\begin{equation*}
v_1\in\left(\ln\varepsilon, \ln\varepsilon+\frac{2|V|}{\mu_{\min} }\right),
\hspace{0.5cm}
v_2\in\left(-2\frac{|V|}{\mu_{\min}}\varepsilon^{\frac{1}{2n}}, -\sqrt{\varepsilon}\right),
\hspace{0.5cm}
v_3\in\left(\sqrt{\varepsilon},\frac{|V|}{\mu_{\min} }\varepsilon^{\frac{1}{2n}}\right),
\end{equation*}
respectively. Similarly, we can draw a schematic diagram for $K(u(x_i))$ as below,
\begin{figure}[H]
  \centering
  \includegraphics[width=10cm,height=6cm]{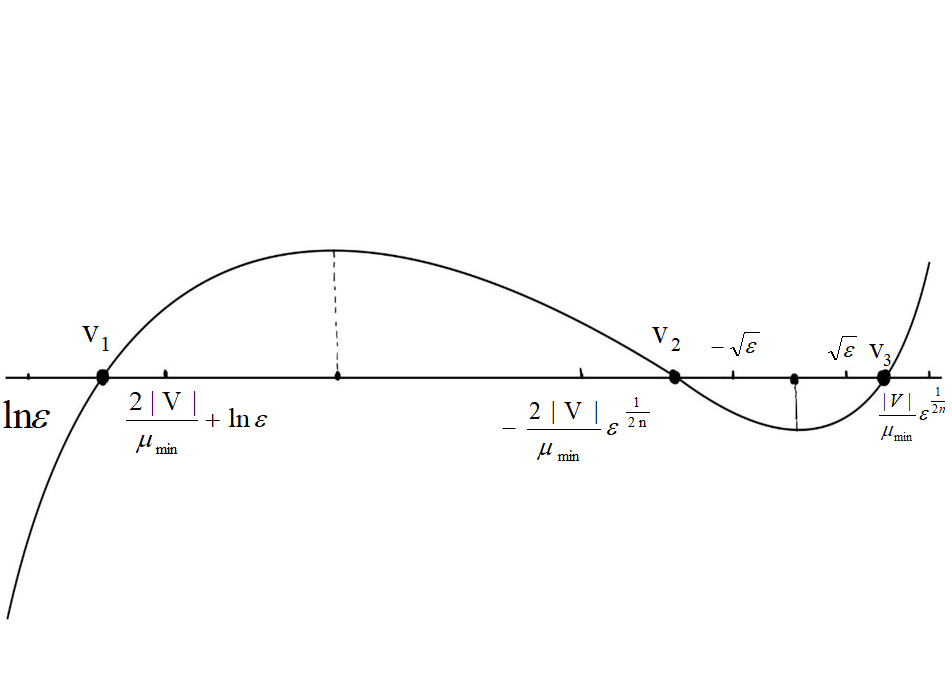}
  \caption{K(u)}\label{figure3}
\end{figure}

\noindent According to Figure \ref{figure3} and (\ref{14}), we deduce that
\begin{equation}\label{15}
u(x_i)\in ( -\infty, v_1 ] \cup [v_2,v_3].
\end{equation}
Based on (\ref{15}), we discuss the cases  when the solution $u$ of (\ref{4}) is located in different intervals.\\

{\bf Case 1.} If every solution $u$ to (\ref{4}) satisfies $u<u_3$, where $u_3$ is the constant solution defined in (\ref{constantsolution}), then we claim that $u<0$.

In fact, we consider  the equation $-\Delta u=\displaystyle\frac{u^{2n}}{1+u^{2n}}e^u -\varepsilon$. Without lose of generality, we set $u(x_0)=\max_V u$. If $u(x_0)=0$, then $\Delta u=\varepsilon>0$, which contradicts to the maximum principle (Lemma 3.1 in \cite{sundegree}). On the other hand, if $u(x_0)>0$, this together with $u(x_0)<u_3$ and
Figure \ref{12} leads to
\begin{equation*}
\Delta (u_3 - u(x_0))
=\displaystyle\frac{u(x_0)^{2n}}{1+u(x_0)^{2n}}e^{u(x_0)}
-
\displaystyle\frac{u_3^{2n}}{1+u_3^{2n}}e^{u_3}<0,
\end{equation*}
which implies $0=\Delta u_3 <\Delta u(x_0)$. It is a contradiction.
Hence, Case 1 gives
\begin{equation}\label{intervel}
  u(x_i)\in ( -\infty, v_1 ] \cup [v_2,0).
\end{equation}

{\bf Subcase 1a} Assume $u$ satisfies the condition in Case 1 and $\#V=m$,
then for any point $x_i\in V$, $u(x_i)$  will only be located in one interval of (\ref{intervel}). Namely,
either $u(x_i)\in ( -\infty, v_1 ]$ for any $i=1,\cdots,m$ or  $u(x_i)\in [v_2,0)$ for any $i=1,\cdots,m$.

To proceed, we prove it by contradiction.
Assume there exists a point $x_i$ such that  $u(x_i)\in ( -\infty, v_1 ]$.
By the connectedness of graph $G=(V, E)$, we have $x_j\in [v_2,0) $, which  satisfy $x_i \sim x_j$ and
\begin{equation}\label{16}
\begin{array}{lll}
-\Delta u(x_j)
=
\displaystyle\frac{u(x_j)^{2n}}{1+u(x_j)^{2n}}e^{u(x_j)}-\varepsilon
\rightarrow 0
\end{array}
\end{equation}
as $\varepsilon\rightarrow 0$. However, the definition of $\Delta$ gives
\begin{align*}\label{17}
\Delta u(x_j)
&=
\frac{1}{\mu_{x_j}}
\left(
\sum_{y\sim{x_j}, u(y)\in (-\infty, v_1 ]}
w_{x_jy}(  u(y)- u(x_j)  )
+
\sum_{y\sim{x_j}, u(y)\in [v_2,0)}
w_{x_jy}(  u(y)- u(x_j)  )
\right)\\[15pt]
&\rightarrow-\infty\nonumber
\end{align*}
as $\varepsilon\rightarrow 0$, which contradicts to (\ref{16}). Thus $u(x_i)$ for any $i=1,\cdots,m$ will only be located in one interval of (\ref{intervel}).\\

{\bf Subcase 1b} If  $u(x)\in ( -\infty, v_1 ]$, then $u(x)\equiv u_1$; If $u(x)\in [v_2,0)$,
then $u(x)\equiv u_2$, where $u_1$ and $u_2$ are constant solutions to the equation (\ref{4}).

In fact, if $u(x)\in ( -\infty, v_1 ]$ and $u(x)\neq u_1$, then note that
\begin{equation*}
\left\{
\begin{array}{lll}
\Delta u=-\displaystyle\frac{u^{2n}}{1+u^{2n}}e^u +\varepsilon,\\[15pt]
\Delta u_1=-\displaystyle\frac{u_1^{2n}}{1+u_1^{2n}}e^{u_1} +\varepsilon,
\end{array}
\right.
\end{equation*}
which implies
\begin{equation*}
\begin{array}{lll}
|\Delta(u-u_1)|
&=
\left|
\displaystyle\frac{u_1^{2n}}{1+u_1^{2n}}e^{u_1}
-
\displaystyle\frac{u^{2n}}{1+u^{2n}}e^u
\right|\\[20pt]
&\leq
\max_V\left\{
\left|\displaystyle\frac{2nu_1^{2n-1}+u_1^{2n}(1+u_1^{2n}) }{ ({1+u_1^{2n}})^2 }\right|e^{u_1},
\left|\displaystyle\frac{2nu^{2n-1}+u^{2n}(1+u^{2n}) }{ ({1+u^{2n}})^2 }\right|e^{u}
\right\} |u-u_1|\\[20pt]
&\leq C\varepsilon |u-u_1|,
\end{array}
\end{equation*}
where in the last step we have used the fact
\begin{equation*}
u(x)\in ( -\infty, v_1 ]\rightarrow-\infty
\hspace{0.2cm}\mathrm{as} \hspace{0.2cm}\varepsilon\rightarrow 0.
\end{equation*}
Thus, it follows from the
elliptic estimate (Lemma 3.2 in \cite{sundegree})
that
\begin{equation*}
\begin{array}{lll}
\max_V (u-u_1)-\min_V (u-u_1)
&\leq C\max_V |\Delta (u-u_1)|\\[15pt]
&\leq C\varepsilon |u-u_1|\\[15pt]
&\leq C\varepsilon
\left(\max_V (u-u_1)-\min_V (u-u_1)+|\min_V (u-u_1)|\right).
\end{array}
\end{equation*}
Let $\varepsilon$ be sufficiently  small such that
$\max_V (u-u_1)\leq \min_V (u-u_1)+\frac{1}{2}|\min_V (u-u_1)|$, then
combining this  inequality  with the equation (\ref{4}) leads to
$\min_V (u-u_1)<0<\max_V (u-u_1)$. Therefore, we obtain
\begin{equation*}
0<\max_V(u-u_1)\leq\min_V (u-u_1)-\frac{1}{2}\min_V (u-u_1)=\frac{1}{2}\min_V (u-u_1)<0,
\end{equation*}
which gives a contradictory estimation. Hence, we get $u(x)\equiv u_1$ when $u(x)\in ( -\infty, v_1 ]$. On the other hand, we can also use the similar idea of the above argument
to conclude $u(x)\equiv u_2$ when $u(x)\in [v_2,0)$.

In conclusion,
we can derive that equation (\ref{4}) only has
constant solutions $u_1\in ( -\infty, v_1 ]$ and $u_2\in [v_2,0)$  when $u<u_3$. Next, we study the behaviour of $u\geq u_3$.\\

{\bf Case 2.} If $u\geq u_3$, then we have $u\equiv u_3$.

For otherwise, considering the equation (\ref{4}), there is
\begin{equation*}
\begin{array}{lll}
0&=-\int_V\Delta u d\mu\\[15pt]
&=\displaystyle\int_V\left(\frac{u^{2n}}{1+u^{2n}}e^u-\varepsilon\right) d\mu\\[15pt]
&>\displaystyle\int_V\left(\frac{u_3^{2n}}{1+u_3^{2n}}e^{u_3}-\varepsilon\right) d\mu\\[15pt]
&=-\displaystyle\int_V\Delta u_3 d\mu\\[15pt]
&=0,
\end{array}
\end{equation*}
which is a contradiction.\\

{\bf Case 3.} If there exist $x_i, x_j\in V$ such that $u(x_i)\leq u_3$ and $u(x_j)\geq u_3$, then we obtain $u\equiv u_3$.

Since $u(x_i)\leq u_3\in(0,v_3]$ and the connectedness of graph $G=(V, E)$, which was used in  Subcase 1a,
we can deduce that $u(x_i)\in [v_2,u_3]$. This combined with the fact that  $u(x_j)\in[u_3,v_3]$ leads to $u(x)\in[v_2,v_3]$.
Furthermore, we have $u(x)\in[v_2,v_3]\rightarrow0$ as $\varepsilon\rightarrow0$. If $u(x)\neq u_3$, then we obtain
\begin{equation*}
\begin{array}{lll}
|\Delta (u-u_3)|
&=
\left|
\displaystyle\frac{u_3^{2n}}{1+u_3^{2n}}e^{u_3}
-
\displaystyle\frac{u^{2n}}{1+u^{2n}}e^u
\right|\\[20pt]
&\leq C\varepsilon |u-u_3|.
\end{array}
\end{equation*}
From these and follows the idea of Subcase 1b, we can derive that $u\equiv u_3$.

To summarise, combining Case 1, Case 2 and Case 3 yield Step 2, namely, the equation
\begin{equation*}
-\Delta u=\displaystyle\frac{u^{2n}}{1+u^{2n}}e^u -\varepsilon,
\end{equation*}
only has three constant solutions. \\

{\bf Step 3.} Let us complete the remaining part of Lemma \ref{degree1},
we calculate the  Brouwer degree of $F(u):=-\Delta u-\displaystyle\frac{u^{2n}}{1+u^{2n}}e^u +\varepsilon $.
Now, for $Q(u):=\displaystyle\frac{u^{2n}}{1+u^{2n}}e^u -\varepsilon$, a direct calculation leads to
\begin{equation}\label{df}
D Q(u)=
\displaystyle\frac{ 2nu^{2n-1}+u^{2n}+u^{4n}}{ (1+u^{2n})^2  }e^u
\end{equation}
and
\begin{equation*}
D^2     Q(u)
=
\displaystyle\frac{ 2nu^{2n-1}+u^{2n}+u^{4n}}{ (1+u^{2n})^2  }e^u
+
\frac{\left(2n(2n-1)u^{2n-2}+2nu^{2n-1}+4nu^{4n-1}\right)(1+u^{2n})-4nu^{2n-1}}{1+u^{2n}}e^u.
\end{equation*}
When  $u$ belongs to the interval in (\ref{constantsolution}), we get
\begin{equation}\label{ddf}
D^2 Q(u)>0,
\end{equation}
which implies that $DQ(u)$ is monotonic. Also, by an easy calculation and for sufficiently  small $\varepsilon>0$, we derive that
if $u_1\in(\ln\varepsilon, 1+\ln\varepsilon)$, then
\begin{equation}\label{u11}
DQ(\ln\varepsilon)
=
\displaystyle\frac{(\ln\varepsilon)^{2n} \left( \displaystyle\frac{2n}{\ln\varepsilon}+1+(\ln\varepsilon)^{2n} \right) \cdot\varepsilon}{\left(1+(\ln\varepsilon)^{2n} \right)^2}>0
\end{equation}
and
\begin{equation}\label{u12}
DQ(1+\ln\varepsilon)
=
\displaystyle\frac{(1+\ln\varepsilon)^{2n} \left(\displaystyle\frac{2n}{1+\ln\varepsilon}+(1+\ln\varepsilon)^{2n} +1\right)  \cdot e \cdot \varepsilon}{\left(1+(1+\ln\varepsilon)^{2n} \right)^2 }
>0.
\end{equation}
Thus, by combining (\ref{ddf}), (\ref{u11}) and (\ref{u12}), we conclude that
\begin{equation}\label{u1}
DQ(u_1)>0.
\end{equation}
On the other hand, if  $u_2\in(-2\varepsilon^{\frac{1}{2n}}, -\sqrt{\varepsilon}),$ then we can obtain

\begin{equation}\label{u21}
DQ(-2\varepsilon^{\frac{1}{2n}} )
=
\displaystyle\frac{\varepsilon^{ 1-\frac{1}{2n} } \left(-n\cdot2^{2n}+2^{2n}\varepsilon^{\frac{1}{2n}}
+2^{4n}\varepsilon^{ 1+\frac{1}{2n} }+n\cdot2^{2n+1}\varepsilon^{\frac{1}{2n} }-2^{2n+1}\varepsilon^{\frac{1}{n}}
-2^{4n+1}\varepsilon^{1+\frac{1}{n}}\right)}
{  \left( 1+2^{2n}\varepsilon \right)^2}<0
\end{equation}
and

\begin{equation}\label{u22}
DQ( -\sqrt{\varepsilon} )
=
\displaystyle\frac{\varepsilon^{n-\frac{1}{2}} \left(-2n+\varepsilon^{\frac{1}{2}}+ \varepsilon^{\frac{1}{2}+n}
+2n\varepsilon^{\frac{1}{2}}-\varepsilon-\varepsilon^{n+1} \right)}
{\left( 1+\varepsilon^n \right)^2 }
<0.
\end{equation}
Now, it follows from (\ref{ddf}) and
(\ref{u21})-(\ref{u22}) that
\begin{equation}\label{u2}
DQ(u_2)<0.
\end{equation}
Similarly, if  $u_3\in(\sqrt{\varepsilon},\varepsilon^{\frac{1}{2n}} )$,  a simple calculation gives
\begin{equation}\label{u31}
DQ(\sqrt{\varepsilon} )
=
\displaystyle\frac{2n(\sqrt{\varepsilon})^{2n-1}+(\sqrt{\varepsilon})^{2n}+(\sqrt{\varepsilon})^{4n}}
{\left(1+\varepsilon^n\right)^2}e^u>0
\end{equation}
and
\begin{equation}\label{u32}
DQ(\varepsilon^{\frac{1}{2n}})
=
\displaystyle\frac{2n\varepsilon^{\frac{2n-1}{2n}}+\varepsilon^2+\varepsilon}{(1+\varepsilon)^2}e^u
>0.
\end{equation}
A combination of (\ref{ddf}) and  (\ref{u31})-(\ref{u32})  yields
\begin{equation}\label{u3}
DQ(u_3)>0.
\end{equation}
Observe that $-\Delta$ is a nonnegative matrix and $0$ is an eigenvalue of $-\Delta$ with multiplicity
one, then  using the homotopy invariance of the Brouwer degree with (\ref{u1}), (\ref{u2}) and (\ref{u3}), we can deduce
\begin{equation*}
\begin{array}{lll}
d_{h,c}
&=\mathrm{sgn\hspace{0.1cm}det}(DF(u_1))+\mathrm{sgn\hspace{0.1cm} det}(DF(u_2))+\mathrm{sgn\hspace{0.1cm} det}(DF(u_3))\\[12pt]
&=\mathrm{sgn\hspace{0.1cm} det}(-\Delta-DQ(u_1))+\mathrm{sgn\hspace{0.1cm} det}(-\Delta-DQ(u_2))+\mathrm{sgn\hspace{0.1cm} det}(-\Delta-DQ(u_3))\\[12pt]
&=-1+1-1\\[12pt]
&=-1.
\end{array}
\end{equation*}
Therefore, the proof of Lemma \ref{degree1} is completed.
\end{proof}

Next, we compute the Brouwer degree for the case  $c=0$, and we state the following lemma.

\begin{Lemma}\label{degree2}
If $c=0$, then for the connected finite graph $G=(V,E)$  with sign changed function $h\in V^{\mathbb{R}}$ satisfying  $\int_V h d\mu<0$, we have $d_{h,c}=-1$.
\end{Lemma}

\begin{proof}
Let $u_t\in V^{\mathbb{R}}$  solve
\begin{equation*}
-\Delta u_t
=
\displaystyle h \frac{u_t^{2n}}{1+u^{2n}_t}e^{u_t}-t, \hspace{0.2cm} \forall t\in [0,1].
\end{equation*}
Firstly, we claim that there exists some positive constant $C>0$ such that
\begin{equation*}
\max_V |u_t|\leq C.
\end{equation*}
Otherwise, applying Lemma \ref{brezis}  yields that  either
$u_t$ converges to $-\infty$ uniformly, or  there exists $x_0\in V$ such that $h(x_0)=0$
and $u_t(x_0)$ converges to $-\infty$. Furthermore, $u_t$ is uniformly bounded in
$\{x\in V:h(x)>0\}$,  but $u_t$ is only uniformly bounded from below in $V$.

If $u_t$ converges to $-\infty$ uniformly, then the proof  of Lemma \ref{compactness}
implies that $u_t-\min_V u_t$ converges to $0$ uniformly. This implies that
\begin{equation*}
0>\int_V h d\mu=\lim\limits_{t\rightarrow0}\displaystyle\int_V h \displaystyle\frac{(u_t-\min_V u_t)^{2n}}{1+(u_t-\min_V u_t)^{2n}}e^{u_t-\min_V u_t}d\mu
=
\lim\limits_{t\rightarrow0}\displaystyle\int_V t e^{-\min_V u_t}d\mu
\geq 0,
\end{equation*}
which leads to a contradiction. For the other case, we arrive at
\begin{equation*}
\begin{array}{lll}
\displaystyle\int_V
|h| \displaystyle\frac{u_t^{2n}}{1+u^{2n}_t}e^{u_t}d\mu
&=
\displaystyle\int_{\{x\in V:h(x)>0\}} h \frac{u_t^{2n}}{1+u^{2n}_t}e^{u_t}d\mu
-
\displaystyle\int_{\{x\in V:h(x)<0\}} h \frac{u_t^{2n}}{1+u^{2n}_t}e^{u_t}d\mu\\[12pt]
&=2\displaystyle\int_{\{x\in V:h(x)>0\}} h \frac{u_t^{2n}}{1+u^{2n}_t}e^{u_t}d\mu
-
\displaystyle\int_{V} h \frac{u_t^{2n}}{1+u^{2n}_t}e^{u_t}d\mu\\[12pt]
&\leq C,
\end{array}
\end{equation*}
then it follows from the elliptic estimate (Lemma 3.2 in \cite{sundegree}) that
\begin{equation*}
\max_V u_t- \min_V u_t  \leq  C\max_V |\Delta u_t|\leq C,
\end{equation*}
which together with $\max_V u_t\geq u(x_0)\rightarrow+\infty $ leads to
$\min_V u_t\rightarrow+\infty $. Combining this
with (\ref{omega}) yields $\Omega=\emptyset$. Namely, we have $h(x)\leq0$, which is a contradiction. Then, applying the homotopy invariance of the Brouwer degree yields
\begin{equation*}
  d_{h,0}=-1.
\end{equation*}
\end{proof}

Finally, we discuss the case $c<0$ and compute the Brouwer degree in this case.

\begin{Lemma}\label{degree3}
If $c<0$,
then for the connected finite graph $G=(V,E)$ with
$h<0$ and $\# V=m$,
we have
\begin{equation*}
d_{h,c}=(-1)^m.
\end{equation*}
\end{Lemma}

\begin{proof}

If $u_t\in V^{\mathbb{R}}$  is a solution of
\begin{equation*}
-\Delta u_t
=\left(-t+(1-t)h\right)\frac{u_t^{2n}}{1+u_t^{2n}}e^{u_t}-(1-t)c+t\varepsilon,
\hspace{0.2cm}  \forall t\in[0,1],
\end{equation*}
then applying Lemma \ref{compactness} yields $\max_V|u_t|\leq C.$ According to the homotopy invariance of the Brouwer degree, we consider
\begin{equation*}
\Delta u
=
\displaystyle  \frac{u^{2n}}{1+u^{2n}}e^{u}-\varepsilon:=Q(u).
\end{equation*}
Set $F(u)= \Delta u - Q(u).$  Similar to the proof of Lemma \ref{degree1}, we can conclude
\begin{equation*}
d_{h,c}=(-1)^m.
\end{equation*}

\noindent This implies the desired result of the lemma.
\end{proof}

\noindent{\bf Remark.}
If $c<0$ and $h$ is sign-changing, we need pay more attention to the existence of solutions to the  equation
\begin{equation}\label{remark}
-\Delta u=\displaystyle h\frac{u^{2n}}{1+u^{2n}}e^u -c.
\end{equation}
In the following content,  we provide some examples to illustrate the existence and non-existence  of  solutions to equation (\ref{remark}) under different assumptions.

{\bf Example 1.}
Without loss of generality, we assume $\mu=w_{xy}= 1$ and $c=-0.01$ with $n=3$.
Let $V=\{x_1,x_2\}$ with $h(x_1)=-5$ and  $h(x_2)=20$.
For convenience, we denote $u(x_1)=x$ and $u(x_2)=y$. Then the equation (\ref{remark}) is equivalent to
the following equation,
\begin{equation*}
\left\{
\begin{array}{lll}
x=y-5\displaystyle\frac{x^{6}}{1+x^{6}}e^{x}+0.01,\\[12pt]
y=x+20\displaystyle\frac{y^{6}}{1+y^{6}}e^{y}+0.01.
\end{array}
\right.
\end{equation*}
the picture of the equation above is showed as blow.
\begin{figure}[H]
  \centering
  \includegraphics[width=13cm,height=5cm]{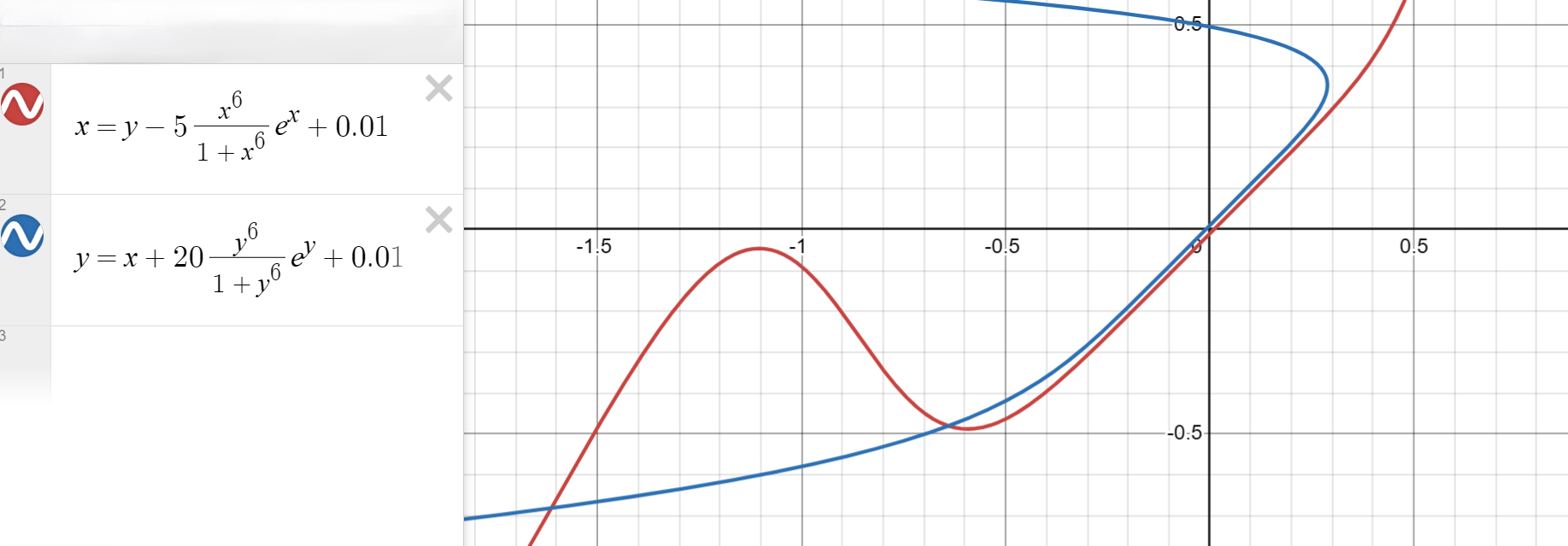}
  \caption{Example 1}\label{examplefigure1}
\end{figure}
From Figure  \ref{examplefigure1}, it is clear  that there exist two solutions to the equation (\ref{remark}) in this situation. Now, we provide an example to demonstrate that (\ref{remark})   does not have a solution in certain special cases.

{\bf Example 2.}
Similarly,  we assume $\mu=w_{xy}= 1$ and $c=-5$ with $n=3$.
Let $V=\{x_1,x_2\}$ with $h(x_1)=-1$ and  $h(x_2)=2$.
For convenience, we denote $u(x_1)=x$ and $u(x_2)=y$. Then the equation  (\ref{remark}) is equivalent to
the following equation,
\begin{equation*}
\left\{
\begin{array}{lll}
x=y-\displaystyle\frac{x^{6}}{1+x^{6}}e^{x}+5,\\[12pt]
y=x+2\displaystyle\frac{y^{6}}{1+y^{6}}e^{y}+5.
\end{array}
\right.
\end{equation*}
and we can obtain the corresponding picture,
\begin{figure}[H]
  \centering
  \includegraphics[width=13cm,height=5cm]{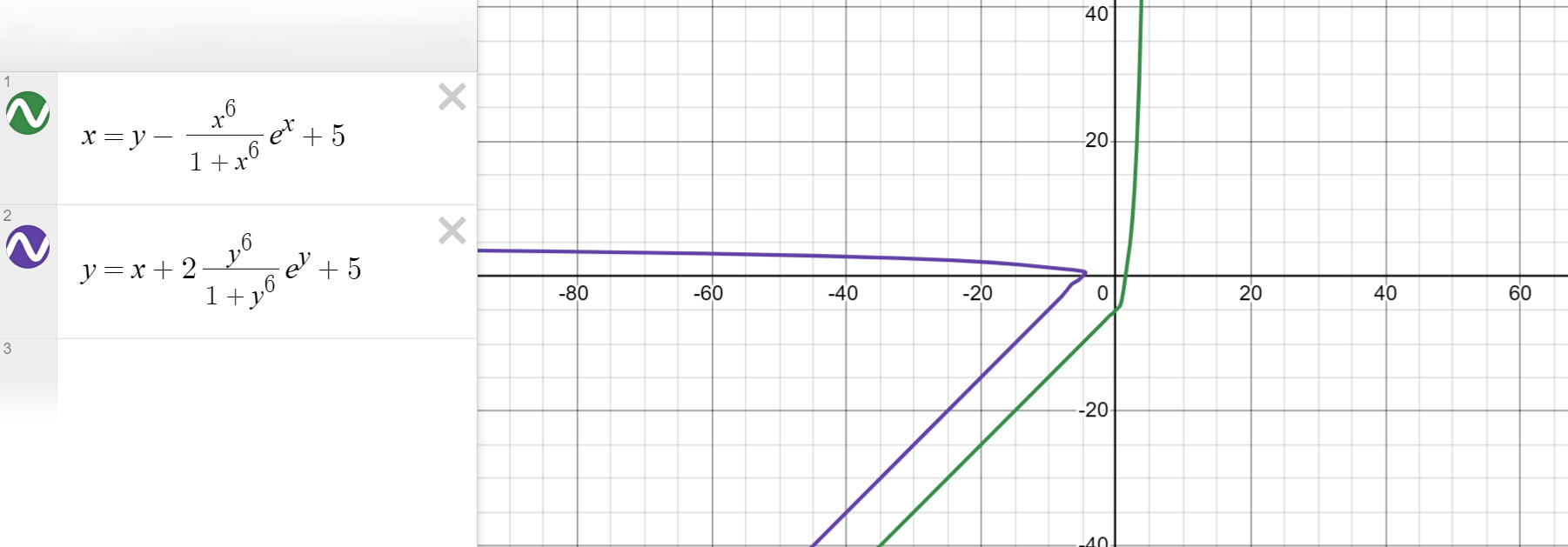}
  \caption{Example 2}\label{examplefigure2}
\end{figure}
\noindent From Figure \ref{examplefigure2}, we know  that the equation (\ref{remark}) has no solution in this situation. Actually, the  non-existence of a solution in Example 2 is also true for $c=0.01$.

Examples 1-2 imply that,
to some extent, it is totally different  from the cases
considered above to find the solutions for (\ref{remark}) when  $c<0$ and sign-changing $h$. Moreover, we observe that the existence and non-existence of
solutions to the equation (\ref{remark}) are closely related to $h$ and $c$. Hence, we need a more careful analysis to study the equation (\ref{remark}), which is interesting  that we will research in the future.

In conclusion, we obtain Theorem \ref{degree}.

{\bf Acknowledgements.}
The author is supported by the Outstanding
Innovative Talents Cultivation Funded Programs 2023 of Renmin University of China.

%

\newpage


\bigskip

\begin{thebibliography}{40}

\bibitem{Caffarellivortex} Caffarelli, L., Yang, Y.: Vortex condensation in the Chern-Simons Higgs model: an existence
theorem, Commun. Math. Phys. 168 (1995) 321–336.



\bibitem{chang} Chang, K.: Methods in nonlinear analysis, Springer Monographs in Mathematics,
Springer-Verlag, Berlin (2005)




\bibitem{ChenScalar} Chen, W., Ding, W.: Scalar curvatures on $S^2$, Trans. Am. Math. Soc. 303 (1987) 365–382.



\bibitem{chenliqua} Chen, W., Li, C.:
Qualitative properties of solutions to some nonlinear elliptic equations in $\mathbb{R}^2$, Duke Math. J. 71 (1993) 427-439.

\bibitem{chenligau} Chen, W., Li, C.: Gaussian curvature on singular surfaces, J. Geom. Anal. 3 (1993) 315-334.

\bibitem{ding} Ding, W., Jost, J., Li, J., Wang, G.: The differential equation
    $\Delta u=8\pi-8\pi he^u$ on a compact Riemann Surface, Asian J. Math. 1 (1997) 230-248.


\bibitem{dingjost}  Ding, W., Jost, J., Li, J., Wang, G.: An analysis of the two-vortex case in the Chern-Siomons Higgs model, Calc. Var. 7 (1998) 87-97.

\bibitem{gaoChern} Gao, J., Hou, S.: Existence theorems for a generalized Chern-Simons equation on finite graphs. J. Math. Phys. 64 (2023) 12 pp.


\bibitem{geneggative} Ge, H.: Kazdan-Warner equation on graph in the negative case, J. Math. Anal. Appl. 453 (2017) 1022–1027. 

\bibitem{geinfinite} Ge, H., Jiang, W.: Kazdan-Warner equation on infinite graphs. J. Korean Math. Soc. 55 (2018) 1091–1101.


\bibitem{yangkw} Grigor’yan, A., Lin, Y.,  Yang, Y.: Kazdan–Warner equation on graph, Calc. Var.
Partial Differential Equations 55 (2016) Paper No. 92, 13 pp.


 \bibitem{yangyamabe} Grigor’yan, A., Lin Y.,  Yang, Y.: Yamabe type equations on graphs, J. Differential Equations 261 (2016) 4924–4943.


\bibitem{yangnonlinear} Grigor’yan, A., Lin Y.,  Yang, Y.: Existence of positive solutions to some nonlinear equations on locally finite graphs, Sci. China Math. 60 (2017) 1311–1324.





\bibitem{weijunchengs6} Gui, C., Li, T., Wei, J., Ye, Z.:
On Beckner's Inequality for Axially Symmetric Functions on $\mathbb{S}^6$, arXiv:2304.04955.




\bibitem{huagroundstate} Hua, B., Xu, W.: Existence of ground state solutions to some nonlinear Schr$\ddot{\mathrm o}$dinger equations on lattice graphs, Calc. Var. Partial Differential Equations 62 (2023) 17 pp.

\bibitem{linyau} Huang, A. Lin, Y., Yau S.:
Existence of Solutions to Mean Field Equations on Graphs, Commun. Math. Phys. 377 (2020) 613–621.







\bibitem{kazdanwarner} Kazdan J., Warner, F.:
    Curvature functions for compact 2-manifolds, Ann. of Math. 99 (1974) 14–47.


\bibitem{kellercompactifiable} Keller, M., Schwarz, M.: The Kazdan-Warner equation on canonically compactifiable graphs, Calc. Var. Partial Differential Equations 57 (2018) 18 pp.

\bibitem{lidegree} Li, J., Sun, L., Yang, Y.: Topological degree for Chern–Simons Higgs models on finite graphs, Calc. Var. Partial Differential Equations 63 (2024) 21 pp.









\bibitem{linCalculus} Lin, Y., Yang, Y., Calculus of variations on locally finite graphs, Rev. Mat. Complut. 35 (2022) 791–813.


\bibitem{liushuangnegetive} Liu, S., Yang, Y.: Multiple solutions of Kazdan-Warner equation on graphs in the negative case, Calc. Var. Partial Differential Equations 59 (2020) 15 pp.



\bibitem{NolascoNontopolo} Nolasco, M.: Nontopological $N$-vortex condensates for the self-dual Chern-Simons theory, Commun. Pure Appl. Math. 56 (2003) 1752–1780.


\bibitem{pansign} Pan, G., Ji, C.: Existence and convergence of the least energy sign-changing solutions for nonlinear Kirchhoff equations on locally finite graphs, Asymptot. Anal. 133 (2023) 463–482.


\bibitem{RicciardiVortices} Ricciardi, T.,  Tarantello, G.:  Vortices in the Maxwell-Chern-Simons theory, Commun. Pure Appl. Math. 53 (2000) 811–851.



\bibitem{shaologarithmic} Shao, M., Yang, Y., Zhao, L.: Multiplicity and limit of solutions for logarithmic Schr$\ddot{\mathrm o}$dinger equations on graphs. J. Math. Phys. 65 (2024) 17 pp.


\bibitem{sundegree} Sun, L., Wang, L.: Brouwer degree for Kazdan-Warner equations on a connected finite graph, Adv. Math. 404 (2022) 29 pp.




\bibitem{yangNormalized} Yang, Y., Zhao, L.: Normalized solutions for nonlinear Schr$\ddot{\mathrm o}$dinger equations on graphs, J. Math. Anal. Appl. 536 (2024) 17 pp.



\bibitem{zhangxiaonegetive} Zhang, X., Chang, Y.: p-th Kazdan-Warner equation on graph in the negative case, J. Math. Anal. Appl. 466 (2018) 400–407.


\end{thebibliography}
\end{document}